\documentclass[reqno]{amsart}

\usepackage{latexsym,amssymb,amsmath, amsfonts}
\usepackage{amsthm}
\usepackage{amscd}
\usepackage[shortlabels]{enumitem}
\usepackage{a4wide}

\numberwithin{equation}{section}
\usepackage[usenames,dvipsnames,svgnames,table]{xcolor}
\usepackage{hyperref}
\hypersetup{colorlinks=true,
	linkcolor=blue,
	urlcolor=blue,
	citecolor=red} 
\newcommand{\R}{\mathbb{R}}

\renewcommand{\le}{\leqslant}
\renewcommand{\ge}{\geqslant}
\renewcommand{\leq}{\leqslant}
\renewcommand{\geq}{\geqslant}

\newcommand{\be}{\begin{equation}}
\newcommand{\en}{\end{equation}}
\newcommand{\ee}{\end{equation}}

\newcommand{\bt}{\begin{theorem}}
\newcommand{\et}{\end{theorem}}
\newcommand{\bp}{\begin{proof}}
\newcommand{\ep}{\end{proof}}
\newcommand{\bc}{\begin{cor}}
\newcommand{\ec}{\end{cor}}
\newcommand{\bl}{\begin{lemma}}
\newcommand{\el}{\end{lemma}}
\newcommand{\bprop}{\begin{prop}}
\newcommand{\eprop}{\end{prop}}

\newcommand{\N}{\mathbb{N}}

\newtheorem{innercustomthm}{Theorem}

\newenvironment{customthm}[1]
{\renewcommand\theinnercustomthm{#1}\innercustomthm}
{\endinnercustomthm}
\newtheorem{theorem}{Theorem}[section]
\newtheorem{remark}{Remark}
\newtheorem{lemma}[theorem]{Lemma}

\newtheorem{corollary}[theorem]{Corollary}
\numberwithin{theorem}{section} \numberwithin{definition}{section}

\newcommand{\RNum}[1]{\uppercase\expandafter{\romannumeral #1\relax}}
\def\R{\mathbb{R}}
\def\N{\mathbb{N}}

\newcommand{\vertiii}[1]{{\left\vert\kern-0.25ex\left\vert\kern-0.25ex\left\vert #1 
		\right\vert\kern-0.25ex\right\vert\kern-0.25ex\right\vert}}

\theoremstyle{definition}

\author[A. Mu\~noz]{Alexander Mu\~noz}
\address{IMECC-UNICAMP, Rua S\'ergio Buarque de Holanda, 651, 13083-859, Cam\-pi\-nas-SP, Bra\-zil}
\email{alexd@ime.unicamp.br}
\author[A. Pastor]{Ademir Pastor}
\address{IMECC-UNICAMP, Rua S\'ergio Buarque de Holanda, 651, 13083-859, Cam\-pi\-nas-SP, Bra\-zil}
\email{apastor@ime.unicamp.br}

\thanks{}


\date{}
\title[Relation decay-regularity of solutions to the mKdV equation]{On the relation between decay and regularity of solutions to the modified Korteweg-de Vries equation}

\begin{document}

\begin{abstract} 
This work is devoted to study the relation between the regularity index  and the decay for solutions of the modified Korteweg-de Vries equation in weighted Sobolev spaces. We show if a solution belongs to a Sobolev space and has a large polynomial decay, compared to the regularity, then  actually it is more regular than initially considered.
\end{abstract}

\maketitle

\section{Introduction}\label{intro}
The existence of solutions to nonlinear dispersive equations have been studied for several years as a direct application of techniques in which the Fourier transform is one of the main ingredients. 
    The best suited space in terms of Fourier transform is the well known Schwartz class in which functions display both, high regularity and fast decay. This relation is also presented by the Fourier transform in which derivatives in space become weights in the frequency space.
    
    One natural question that arises is if a certain decay is preserved by the flow of a dispersive model. Here, we will investigate the relation between the regularity and the decay of solutions. Particular attention will be given to the following Cauchy problem
\begin{equation} \label{ivp}
	\begin{cases}
		\partial_tu+\partial_x^3 u+u^k\partial_x u=0, \quad x\in\mathbb{R},\;t\in\mathbb{R},\\
		u(x,t)=u_0(x),
	\end{cases}
\end{equation}
where $k=1,2$.  If $k=1$ we have the Kortweg-de Vries (KdV) equation and in case $k=2$ we obtain the modified KdV (mKdV) equation. The KdV equation was derived in \cite{kort} as a model describing the propagation of waves in one dimensional dispersive media. This equation has been widely studied in the literature, see for instance \cite{coliander}, \cite{KATO}, \cite{KPV} and the references therein. 

Our main interest here will be on the mKdV equation. Local and global well-posedness in $H^s(\R)$, $s\geq1/4$, for such equation was obtained in \cite{coliander}, \cite{critical}, \cite{KPV}, \cite{kishimoto}. Recently, this index was pushed down in \cite{visan}, where the authors established the (sharp) global  well-posedness in $H^s(\R)$, for $s>-1/2$ in the sense that the solution map extends uniquely from Schwartz space
to a  continuous map in $H^s(\R)$.

Denote with $U(t)f$ the unitary group associated to the linear part of \eqref{ivp}, that is,
\begin{equation}
    \label{semigrupo}
    U(t)f(x):= \left( e^{i t \xi^3} \widehat{f}(\xi)\right)^\vee (x).
\end{equation}
Taking into account the Duhamel formulation of solutions \begin{equation}
    \label{integralequation}
    u(x,t)=U(t)u_0(x)-\int_0^t U(t-t')u(x,t')^k\partial_x u(x,t')dt',
\end{equation} to study expressions of the form $|x|^bu(x,t)$ it is necessary to understand the relation between the unitary group $U(t)$ and the weights $|x|^b$. In \cite{KATO}, using the operators $\Gamma_t:=x-3t\partial_x^2$ and $L:=\partial_t+\partial_x^3$, Kato showed that $[\Gamma_t,L]=0$ and used this to note that both, $U(t)xu_0$ and $\Gamma_tU(t)u_0$, are solutions of the problem
\begin{equation*}
\begin{cases}
    Lu=0,\\u(0)=xu_0.
\end{cases}\end{equation*}
Therefore, it follows that $U(t)xu_0=\Gamma_tU(t)u_0$. The latter can be translated into the formula
\begin{equation}\label{pointwisetroca}
    xU(t)u_0(x)=U(t)xu_0(x)+3tU(t)(\partial_x^2 u_0)(x).
\end{equation}
From this, it is expected that the relation between the regularity index $s$ and the decay $b$ of a solution, at least in the case where these parameters are integer numbers, is $s\ge 2b$. 

Recently,  in \cite{FLP}, using one version of the Stein derivative (which does not directly depend on the Fourier transform), the formula in \eqref{pointwisetroca} was generalized to weights with fractional powers. It was proven that for $b\in(0,1)$ we have
\begin{equation}\label{temp}
        |x|^bU(t)u_0(x)=U(t)(|x|^bu_0)(x)+U(t)\{\Phi_{t,b}(\widehat{u}_0(\xi)) \}^\vee (x),
    \end{equation}
where the residual term $\{\Phi_{t,b}(\widehat{u}_0(\xi)) \}^\vee$ is well behaved in $L^2(\R)$ in the sense detailed in Theorem \ref{troca} below. Then, we can also expect the relation $s\ge 2b$ for weights with fractional powers. 

For a formulation analogous to \eqref{temp} for a wide family of dispersive models and several dimensions we refer the reader to Theorem 1.1 in \cite{AA}.

The main topic of this work is then to establish the fact the relation $s\ge 2b$ is optimal for the mKdV equation in the sense that if the solution has a decay that exceeds $s/2$, say, $b=s/2+\varepsilon$, then  actually the solution is more regular than initially considered. In particular, the index of regularity would be $s+2\varepsilon$. This will be cleared up in the statement of our main result and its consequences. (See for instance Corollary \ref{color} below).

The first result in this direction is due to Isaza, Linares and Ponce \cite{ILP2013} for the KdV equation. It reads as follows:
\begin{customthm}{A}\label{ilp}
    Let $u\in C(\R, L^2(\R))$ be the global solution  of the IVP \eqref{ivp} with $k=1$. If there exist $\alpha>0$ and two different times $t_0,\ t_1\in \R$ such that 
    \begin{equation*}
        |x|^\alpha u(x,t_0),\ |x|^\alpha u(x, t_1) \in L^2(\R),
    \end{equation*}
    then $u\in C(\R, H^{2\alpha}(\R))$.
\end{customthm}

An application of the same principle used in the proof of Theorem \ref{ilp} to the fifth-order KdV equation \begin{equation*}
    \partial_t u+\partial_x^5 u + u\partial_x u =0
\end{equation*} was done by Bustamante, Jim\'enez and Mej\'ia in \cite{5kdv}.

The proof of Theorem \ref{ilp} consists in a bootstrap argument depending on the size of $\alpha$. In each step the authors perform energy estimates in an accurate way to obtain decay for the solution and its derivative. For instance, in the first step  (assuming $\alpha\in(0,1/2]$) they established that for almost every $t\in [t_0,t_1]$ it follows that $\partial_x u(t)\langle x\rangle^{\alpha-1/2}\in L^2(\R)$ and for all $t\in [t_0,t_1]$, $u(t)\langle x\rangle^{\alpha}\in L^2(\R)$. By considering $f:=\langle x\rangle^{\alpha-1/2} u(t^*)$, for some $t^*\in[t_0,t_1]$, the latter implies that $Jf$ and $\langle x\rangle^{1/2}f$ are in $L^2(\R)$. Using interpolation, the following bound was obtained:
\begin{equation}
    \label{unito} \left\|J^{\theta}\left( \langle x\rangle^{(1-\theta)\frac{1}{2}}f \right) \right\|_2\le c\|Jf\|_{2}^{\theta} \|\langle x\rangle^{1/2}f\|_2^{1-\theta}, \ \ \ \ \theta\in (0,1).
\end{equation}
By requiring that $(1-\theta)/2 + \alpha -1/2 =0$, that is, $\theta=2\alpha$ they proved that $u(t^*)\in H^{2\alpha}(\R)$. The idea is then to repeat the above procedure in each step with an interval of length $1/2$ for $\alpha$.



\subsection{Main Result}\hfill

Based on the proof of Theorem \ref{ilp} and taking into account some ideas from Theorem 1.6 in \cite{5kdv}, we establish the following result.

\begin{theorem}\label{teo0}
Let $v_0\in H^{1/4}(\R)$. Let $v\in C\left([0,T];H^{1/4}(\R)\right)$ be the solution of the IVP \eqref{ivp} with $k=2$ provided by Theorem \ref{localtheorymkdv}. Assume there exist $t_0, t_1\in[0,T]$ with $t_0<t_1$ and $\alpha>0$ such that \begin{equation}
    \label{twotimes} |x|^\alpha v(t_0)\in L^2(\R) \mbox{ and } |x|^\alpha v(t_1) \in L^2(\R). 
\end{equation} Then $v\in C\left([0,T]; H^{2\alpha}(\R)\right)$.
\end{theorem}

Note that in Theorem \ref{teo0} when $\alpha\in(0,1/8]$ there is no gain of extra regularity. This is consistent with the weighted local theory in which the solution persists in $H^{1/4}(\R)\cap L^2(|x|^{2b}dx)$ whenever $b\le 1/8$. Another way to formulate Theorem \ref{teo0} is to replace \eqref{twotimes} with 
    \begin{equation*}
        |x|^{1/4+\alpha} v(t_0)\in L^2(\R) \mbox{ and } |x|^{1/4+\alpha} v(t_1) \in L^2(\R)
    \end{equation*}
    and the conclusion $v\in C\left( [0,T]; H^{2\alpha}(\R) \right)$ with $v\in C\left( [0,T]; H^{\frac{1}{4} + 2\alpha} (\R)\right)$ instead.
    
    \begin{remark}
    	We may interpret Theorem \ref{teo0}  as an ill-posedness result in the following sense: assume  $v_0$ belongs to  $Y_{b}:=\left(H^{1/4}(\R)\setminus H^{1/4^+}(\R)\right)\cap L^2(|x|^{2b}dx)$, for $b\in \R$ with $b> 1/8$, then the corresponding solution of the mKdV equation does not belong to $Y_{b}$.
    \end{remark}

In \cite{ILP2013} the authors suggested that the proof of Theorem \ref{ilp} extends to the mKdV. It appears that some adjustments to the bootstrap argument are required to rise according to the size of $\alpha$. For instance, if $\alpha>1/2$ the same choice of $\theta$ in \eqref{unito} might not be the best due to the constrain $\theta\le 1$. Moreover, regardless of the choice of $\theta$, the most regular scenario for $J^\theta$ leads to $H^1(\R)$.

This concern exhibits the needing of increase the regularity over $f$, that is, to get a higher estimate than $J^1f$; while the role of $\alpha$ is re-escalated to fit $(1-\theta)/2 + \tilde{\alpha}-1/2=0$ (for some $\tilde{\alpha}$) with $\theta\in(0,1)$. One natural way of doing this is as follows: in case $\alpha\in (\frac{r}{2},\frac{r+1}{2}]$ we can take $\partial_x^r$ to the mKdV equation and prove that $\partial_x^{r+1} u(t)\langle x \rangle^{\tilde{\alpha}-1/2}$ is in $L^2(\R)$, where $\tilde{\alpha}=\alpha-r/2\in(0,1/2]$.
By setting $f:=\langle x\rangle^{\tilde{\alpha}-1/2} u(t^*)$ it might be seen that $J^{r+1}f$ and $\langle x\rangle^{1/2}f$ are in $L^2(\R)$. Interpolating as in \eqref{unito} we would get
\begin{equation*}
     \left\|J^{(r+1)\theta}\left( \langle x\rangle^{(1-\theta)\frac{1}{2}}f \right) \right\|_2\le c\|J^{r+1}f\|_{2}^{\theta} \|\langle x\rangle^{1/2}f\|_2^{1-\theta}, \ \ \ \ \theta\in (0,1).
\end{equation*}
From the condition $(1-\theta)/2 + \tilde{\alpha}-1/2=0$ we obtain $\theta=2\alpha-r$, which is in $(0,1)$. Unfortunately this interpolation will not lead to the best estimate for $J^s u$ because $(r+1)\theta$ is less than or equal to the expected gain of regularity of $2\alpha$ and only attains it when $\alpha=(r+1)/2$. 

Given that only an increase in the regularity does not resolve on its own, in the proof of our main theorem we not only increase the spatial derivatives of $u$ but we also consider some decay for them. We manage to prove that in the general case $\alpha\in (\frac{r}{2},\frac{r+1}{2}]$, for almost every $t$ in a subinterval of $[t_0,t_1]$ we have
\begin{equation}
    \label{bootstrap1}
    \langle x\rangle^{\tilde{\alpha}-1/2}\partial_x^{r+1}u(t) \in L^2(\R) \hspace{5mm}\mbox{and}\hspace{5mm} \langle x\rangle^{\tilde{\alpha}}\partial_x^{r}u(t) \in L^2(\R),
\end{equation}
By considering now $f:=\langle x\rangle^{\tilde{\alpha}-1/2}\partial_x^{r}u(t^*)$ (for some $t^*\in[t_0,t_1]$) it is noted that $Jf$ and $\langle x\rangle^{1/2}f$ are in $L^2(\R)$. Therefore interpolating analogously to \eqref{unito} with $\theta=2\tilde{\alpha}$ we would obtain $J^{2\tilde{\alpha}}\partial_x^{r}u(t^*)\in L^2(\R)$, which is to say $u(t^*)\in H^{2\alpha}(\R)$.

This paper is organized as follows: in Section \ref{presection} we introduce some notation and give the preliminary theory in order to prove Theorem \ref{teo0}. In Section \ref{theorem1.1proof} we breakdown the proof and enunciate some consequences. 

\section{Preliminaries}\label{presection}
\subsection{Notation}\hfill

We use $c$ and $M$ to denote several constants that may vary from line to line. Dependence on parameters is indicated using subscripts. We write $a\sim b$ when there exists a constant $c>0$ such that $a/c\le b\le ca$. For a real number $s$ we denote with $s^+$ the quantity $s+\varepsilon$ for an arbitrary small $\varepsilon>0$.

The set $L^p(\R)$ is the usual Lebesgue space and $L^2(wdx)$ represent the space $L^2$ with respect to the measure $w(x)dx$. The norm in $L^p(\R)$ will be denoted by $\|\cdot\|_p$. For a function $f$, the expressions $\widehat{f}$ and $f^\vee$ mean the Fourier and inverse Fourier transform, respectively.  For $s, b\in \R$ the set $H^s(\R)$ is the $L^2$-based Sobolev space of index $s$ and $Z_{s,b}$ is the weighted Sobolev space $H^s(\R)\cap L^2(|x|^{2b}dx)$. 
We denote with $\|\cdot\|_{L^p_xL^q_T}$ the norm in the mixed Lebesgue spaces defined by 
\begin{equation*}
    \|f\|_{L^p_xL^q_T}:= \left\| \| f(\cdot,\cdot) \|_{L^q_T} \right\|_{L^p_x}.
\end{equation*}
We use $D^s_x$ for the classical fractional derivative $\widehat{D_x^sf}(\xi)=|\xi|^s \widehat{f}(\xi)$. For simplicity we adopt the notation $\langle x\rangle:=(1+x^2)^{1/2}$. Also, denote with $J^sf$ the potential $\widehat{J^s(f)}(\xi)=\langle \xi \rangle^{s}\widehat{f}(\xi).$
\subsection{Linear estimates}\hfill

We first introduce some results related to the group $U(t)$ defined in \eqref{semigrupo}.

\begin{lemma}
    For $u_0\in L^2(\R)$ we have
    \begin{equation}
        \label{ww1}
        \|D^{-1/4}_x U(t)u_0\|_{L^4_xL^\infty_T}\le c \|u_0\|_{L^2_x}
    \end{equation}
\end{lemma}
\begin{proof}
    See Theorem 3.7 in \cite{KPV}.
\end{proof}

\begin{theorem}\label{troca}
    Let ${b}\in(0,1)$. If $u_0\in Z_{2{b},{b}}$ then for all $t\in \mathbb{R}$ and for almost every $x\in \mathbb{R}$
    \begin{equation}
        \label{troca2}
        |x|^{b} U(t)u_0(x)=U(t)(|x|^{b} u_0)(x)+U(t)\{\Phi_{t,{b}}(\widehat{u}_0(\xi)) \}^\vee (x),
    \end{equation}
with 
\begin{equation}\label{residualestimate}
    \|U(t)\{\Phi_{t,{b}}(\widehat{u}_0(\xi)) \}^\vee  \|_2\le c(1+|t|)\|u_0\|_{H^{2{b}}}.
\end{equation}
Moreover, if in addition one has that for $\beta\in (0,{b})$,
\begin{equation*}
    D^\beta(|x|^{b} u_0)\in L^2(\R) \ \mbox{and} \ u_0\in H^{\beta+2{b}}(\R),
\end{equation*}
then for all $t\in \R$ and for almost every $x\in \R$,
\begin{equation*}
    D^\beta_x(|x|^{b} U(t)u_0)(x)=U(t)(D^\beta|x|^{b} u_0)(x) +U(t)(D^\beta (\left\{\Phi_{t,{b}}(\widehat{u}_0)(\xi) \right\}^\vee)(x),
\end{equation*}
with
\begin{equation}\label{residualderiv}
    \|D^\beta (\left\{\Phi_{t,{b}}(\widehat{u}_0)(\xi) \right\}^\vee) \|_2 \le c(1+|t|)\|u_0\|_{H^{\beta+2{b}}}.
\end{equation}
\end{theorem}
\begin{proof}
    See Theorem 1 in \cite{FLP}.
\end{proof}

Next, we state an interpolation lemma relating regularity and decay that was first introduced in Lemma 4 of \cite{NAHASPONCE}. 

\begin{lemma}\label{lemma}
    Let $a, b>0$. Assume that $J^{a} f$ and $\langle x\rangle^{b} f$ are in $L^{2}(\mathbb{R}) .$ Then for any $\theta \in(0,1)$,
$$
\left\|J^{\theta a}\left(\langle x\rangle^{(1-\theta) b} f\right)\right\|_{2} \leq c\left\|\langle x\rangle^{b} f\right\|_{2}^{1-\theta}\left\|J^{a} f\right\|_{2}^{\theta}.
$$
\end{lemma}
In the next Lemma, the class $A_p$ denotes the Muckenhoupt class on $\R$, that is, the set of all the weights $w$ such that 
\begin{equation*}
    [w]_p=\sup\limits_I\left(\frac{1}{|I|}\int_I w(y)dy \right)\left(\frac{1}{|I|}\int_I w(y)^{-\frac{1}{p-1}}dy \right)^{p-1}<\infty,
\end{equation*}
where the supremum is taken over all intervals $I\subset \R$ (see \cite{cruz} and \cite{hunt} for details).
\begin{lemma}
    \label{HA2}
    The Hilbert transform is bounded in $L^p(wdx)$, $1<p<\infty$, if and only if $w\in A_p$. 
\end{lemma}
\begin{proof}
    See Theorem 9 in \cite{hunt}.
\end{proof}

\subsection{Local theory results}\hfill

We gather some useful results regarding the existence of solutions to \eqref{ivp} in Sobolev and weighted Sobolev spaces. 

\begin{theorem}[Local theory]\label{localtheorymkdv}
Let $s\ge 1/4$. Then for any $v_0\in H^{s}(\R)$ there exists $T=T(\|D^{1/4}_xv_0\|_{L^2})=c(\|D^{1/4}_xv_0\|_{L^2}^{-4})$ and a unique solution $v(t)$ of the initial value problem
\begin{equation}\label{mKdV}
    \begin{cases}
        \partial_t v +\partial_x^3 v +v^2\partial_x v =0,\ \ x,t\in \R. \\ v(x,0)=v_0(x)
    \end{cases}
\end{equation}
such that \begin{equation}\label{class1uno}
    v\in C\left([0,T];H^s(\R) \right),
\end{equation}
\begin{equation}\label{class2dos}
    \|D^s_x\partial_xv\|_{L^\infty_xL^2_T}+\|D_x^{s-1/4}\partial_xv\|_{L^{20}_xL^{5/2}_T}+\|D^s_xv\|_{L^5_xL^{10}_T}+\|v\|_{L^4_xL^\infty_T}<\infty
\end{equation}
and \begin{equation}\label{class3tres}
     \|\partial_x v\|_{L^\infty_xL^2_T}<\infty. 
\end{equation}
Moreover, for any $T'\in (0,T)$,  there exists a neighborhood $\mathcal{V}$ of $v_0$ in $H^s(\R)$ such that the map $\Tilde{v}_0\mapsto\tilde{v}(t)$ from $\mathcal{V}$ into the class defined by \eqref{class1uno}-\eqref{class3tres}, with $T'$ instead of $T$, is smooth. 
\end{theorem}
We briefly mention a few aspects related to the proof of Theorem \ref{localtheorymkdv}; for the details we refer the reader to  \cite[Theorem 2.4]{KPV}. The authors used the Banach fixed point theorem in the space
\begin{equation}\label{xta}
    \begin{split}
        X^T_a:=\left\{v\in C([0,T];H^s(\R)) \mid \Lambda^T(v)\le a \right\},
    \end{split}
    \end{equation}
    where 
    \begin{equation*}
        \begin{split}
            \Lambda^T(v)&=\|v\|_{L^\infty_T H^s}+\|D^s \partial_x v\|_{L^\infty_xL^2_T}+\|D^{s-1/4}_x \partial_x v\|_{L^{20}_xL^{5/2}_T}\\& \hspace{5mm}+\|D^s_x v\|_{L^5_xL^{10}_T}+\|v\|_{L^4_xL^\infty_T}+ \|\partial_x v\|_{L^\infty_xL^2_T}.
        \end{split}
    \end{equation*}
It was shown that, for $a=2c\|v_0\|_{H^{1/4}}$ and $T$ such that $2ca^2T^{1/2}<1$, the operator
\begin{equation*}
    \Psi_{v_0}(v)=U(t)v_0-\int_0^tU(t-t')(v^2\partial_x v)(t')dt'
\end{equation*}
has a unique fixed point in $X^T_a$. In particular, it was proved that \begin{equation}
    \label{localcomp1} 
   \int_0^T \|J^s(v^2\partial_x v)\|_{L^2_x}dt' \le  cT^{1/2}\Lambda^T(v)^3,
\end{equation}
which is the estimate for the nonlinear term.

Concerning global well-posedness, we have following result.

\begin{theorem}
    \label{global} The IVP \eqref{ivp} with $k=2$ is globally well-posed for initial data $u_0\in H^s(\R)$, $s\ge 1/4$.
\end{theorem}
\begin{proof}
    See Theorem 3 in \cite{coliander} and Theorem 1.2 in \cite{critical}.
\end{proof}

Since we are interested in weighted Sobolev space, we may combine Theorem \ref{localtheorymkdv} with weighted estimates to obtain

\begin{theorem}[Local  theory in weighted spaces]\label{FLP2015}
Let $v\in C([0,T];H^{s}(\R))$ denote the solution of the IVP \eqref{mKdV} provided by Theorem \ref{localtheorymkdv}. If $v_0\in Z_{s,r}$ with $s\ge 2r$ then the solution satisfies \begin{equation}
    \label{class4}
    v\in C([0,T]; Z_{s,r}) \hspace{3mm} \mbox{and}\hspace{3mm} \||x|^{s/2}v\|_{L^5_xL^{10}_T}<\infty.
\end{equation}
For any $T'\in (0,T)$ there exists a neighborhood $\mathcal{V}$ of $v_0$ in $Z_{s,r}$ such that the map $\tilde{v}_0\mapsto\tilde{v}(t)$ from $\mathcal{V}$ into the class defined by \eqref{class1uno}-\eqref{class3tres} and \eqref{class4} with $T'$ instead of $T$ is smooth. 
\end{theorem}
Essentially, the proof follows in a similar fashion as the one in Theorem \ref{localtheorymkdv} by considering the norm
\begin{equation}\label{class4p}
    \mu^T(v):= \Lambda^T(v)+ \||x|^{s/2}v\|_{L^5_xL^{10}_T} +\||x|^{s/2}v\|_{L^\infty_T L^2_x}
\end{equation}
and using Theorem \ref{troca}; we refer the reader to Theorem 1 in \cite{NAHASPONCE2} for details.

 

\section{Proof of Theorem \ref{teo0}}\label{theorem1.1proof}
The proof is presented in four cases, in which the first three are required to be done explicitly because of technical details involving the size of the truncated weights with respect to the regularity of the solution. The fourth case provides a construction of the solution in a general setting.
\begin{proof}[Proof of Theorem \ref{teo0}]
Without loss of generality we assume $t_0=0$. 
\subsection{Case $\alpha\in(0,1/2]$}\hfill

Let $\{v_{0m}\}_m\subset C_0^\infty(\R)$ be a sequence converging to $v_0$ in $H^{1/4}(\R)$. Denote with $v_m(\cdot)\in H^\infty(\R)$ the corresponding solution provided by Theorem \ref{localtheorymkdv} with initial data $v_{0m}$. By regularity of the data-solution map we can assume all the $v_m$'s are defined in $[0,T]$ with $v_m$ converging to $v$ in $C\left([0,T];H^{1/4}(\R)\right)$.

For $N\in \N$, denote 
\begin{equation*}
    \Tilde{\Tilde{\phi}}_{0,N}^\alpha (x):=
    \begin{cases}
        \langle x \rangle^{2\alpha}-1 \hspace{5mm}  x\in [0,N],\\
        (2N)^{2\alpha} \hspace{9mm}  x\in [3N,\infty).
    \end{cases}
\end{equation*}
Let $\Tilde{\phi}_{0,N}^\alpha$ be a smooth regularization of $\Tilde{\Tilde{\phi}}_{0,N}^\alpha$, defined in $[0,\infty)$, such that for $j=1,2,\dots$ we have $\left|\partial_x^j \Tilde{\phi}_{0,N}^\alpha\right|\le c$ with $c$ independent of $N$. Set $\phi_N\equiv \phi_{0,N}^\alpha$ to be the odd extension of $\Tilde{\phi}_{0,N}^\alpha$, that is, \begin{equation}
    \label{c0}
    \phi_N(x):=\begin{cases}
        \Tilde{\phi}_{0,N}^\alpha(x) \hspace{11mm}  x\in [0,\infty],\\
        -\Tilde{\phi}_{0,N}^\alpha(-x) \hspace{5mm}  x\in (-\infty,0).
    \end{cases}
\end{equation}
Note that $\phi_N^\prime\ge 0$.

Take the mKdV equation for $v_m$ and multiply it by $v_m \phi_N$. After integration by parts in space we get:
\begin{equation}\label{c0s}
    \frac{1}{2}\frac{d}{dt}\int v_m^2 \phi_N dx +\frac{3}{2}\int (\partial_x v_m)^2 \phi_N^\prime dx -\frac{1}{2}\int v_m^2 \phi_N^{(3)}dx -\frac{1}{4}\int v_m^4 \phi_N^\prime dx =0.
\end{equation}
Integrating over $[0,t_1]$ we get:
\begin{equation}
    \label{c1}
    \int v_m^2(t_1)\phi_N dx -\int v_{0m}^2 \phi_N dx + 3 \int_{0}^{t_1}\int (\partial_xv_m)^2\phi_N^\prime dxdt -\int_0^{t_1}\int v_m^2\phi_N^{(3)}dxdt-\frac{1}{2}\int_0^{t_1}\int v_m^4\phi_N^\prime dxdt
=0.    
\end{equation}
For $m$ large enough we have
\begin{equation}
    \label{c2}
    \left|-\int_0^{t_1}\int v_m^2\phi_N^{(3)}dxdt \right| \le c\int_0^{t_1}\int v^2_m dxdt\le c\|v_m\|_{L^2_{x{t_1}}}^2\le c\|v_{0m}\|_{L^2_{xt_1}}^2\le 2ct_1\|v_0\|_{2}^2.
\end{equation}
Also, 
\begin{equation}
    \label{c3}
    \left|-\frac{1}{2}\int_0^{t_1}\int v^4_m\phi_n^\prime dxdt \right|\le c\int_0^{t_1}\int v_m^4 dxdt \le ct_1\|v_m\|^4_{L^4_x L^\infty_{t_1}}\le c\|v_m\|_{L^4_xL^\infty_{t_1}}^4\le 2\|v\|_{L^4_xL^\infty_{t_1}}^4.
\end{equation}
From \eqref{c1}-\eqref{c3} we get that for $m\gg1$,
\begin{equation}\label{c3s}
\begin{split}
    \int_0^{t_1}\int (\partial_x v_m)^2 \phi_N^\prime dxdt& \le c(\|v_{0}\|^2_{2}+\|v\|^4_{L^4_x L^\infty_{t_1}})+\|v_{0}\phi_N^{1/2}\|_{L^2_x}^2+\|v(t_1)\phi_N^{1/2}\|_{L^2_x}^2\\ & \le c(\|v_{0}\|^2_{2}+\|v\|^4_{L^4_x L^\infty_{t_1}}+\|\langle x\rangle^{\alpha}v_{0}\|_{L^2_x}^2+\|\langle x\rangle^{\alpha}v(t_1)\|_{L^2_x}^2).
    \end{split}
\end{equation}
Thus 
\begin{equation}
    \label{c4}
    \limsup\limits_{m\to\infty}\int_0^{t_1}\int (\partial_xv_m)^2\phi_N^\prime dxdt \le M,
\end{equation}
where $M>0$ depends on $\|v_{0}\|_{Z_{1/4,\alpha}}$, $\|v(t_1)\langle x\rangle^{\alpha}\|_{L^2_x}$ and $\|v\|_{L^4_x L^\infty_{t_1}}$.

We claim that, for any fixed $N\in\N$, the left-hand side of \eqref{c3s} actually converges to $$\int_0^{t_1}\int (\partial_x v)^2\phi_N^\prime dxdt$$ as $m\to\infty$. Indeed,
\begin{equation}\label{c5}
\begin{split}
    \left|\int_0^{t_1}\int \left[(\partial_xv_m)^2-(\partial_xv)^2 \right]\phi_N^\prime dxdt \right|&\le c\int_0^{t_1}\int_{[-3N,3N]} |\partial_xv_m -\partial_x v||\partial_x v_m+\partial_xv|dxdt\\&\hspace{-10mm} \le c\|\partial_x v_m-\partial_x v\|_{L^2_{t_1}L^2_x[-3N,3N]}\|\partial_x v_m + \partial_x v\|_{L^2_{t_1}L^2_x[-3N,3N]}\\
    &\hspace{-10mm} \le c_N \|\partial_x v_m-\partial_x v\|_{L^\infty_xL^2_{t_1}}\left(\|\partial_x v_m\|_{L^\infty_xL^2_{t_1}}+\|\partial_x v\|_{L^\infty_xL^2_{t_1}} \right)\\
   & \hspace{-3mm} \xrightarrow[m\to\infty]{}0,
\end{split}
\end{equation}
where we used \eqref{class3tres} together with the continuous dependence on the initial data. 

In view of \eqref{c4}, we therefore have 
\begin{equation*}
    \int_0^{t_1}\int (\partial_x v)\phi_N^{\prime} dxdt \le M.
\end{equation*}

Since $\phi_N^\prime$ is even and for $x>1$ we have that $\phi_N^\prime(x)\xrightarrow[N\to\infty]{}2\alpha \langle x\rangle^{2\alpha-2}x\sim \langle x\rangle^{2\alpha-1}$, we deduce
\begin{equation}
    \label{c6}
    \int_0^{t_1}\int_{|x|>1} (\partial_x v)^2\langle x\rangle^{2\alpha-1}dxdt \le c \liminf\limits_{N\to \infty}\int_0^{t_1}\int_{|x|>1} (\partial_x v)^2\phi_N^\prime dxdt\le M.
\end{equation}
Moreover, 
\begin{equation}
    \label{c7}
    \int_0^{t_1}\int_{|x|<1} (\partial_x v)^2\langle x\rangle^{2\alpha-1} dxdt =\int_{|x|<1}\int_0^{t_1} (\partial_x v)^2\langle x\rangle^{2\alpha-1} dtdx \le c\|\partial_x v\|_{L^\infty_xL^2_{t_1}}^2< \infty.
\end{equation}

From \eqref{c6} and \eqref{c7} we conclude $$\int_0^{t_1}\int (\partial_x v)^2 \langle x\rangle^{2\alpha-1}dxdt<\infty,$$
which implies
\begin{equation}
    \label{c8} \langle x\rangle^{\alpha-1/2}\partial_x v \in L^2(\R) \mbox{  for almost every } t\in[0,t_1].
\end{equation}

Arguing in a similar fashion as done to obtain \eqref{c8} but using $\phi_N$ as the even extension of $\Tilde{\phi}_N^\alpha$ instead, it can be seen that for any $t\in[0,t_1]$ we have $\langle x\rangle^{\alpha}v\in L^2(\R)$ (see also \cite[page 143]{ILP2013}).

Set $t_*\in [0,t_1]$ so that $\langle x\rangle^{\alpha-1/2}\partial_x v(t_*)\in L^2(\R)$. By writing $f:=\langle x\rangle^{\alpha-1/2}v(t_*)$, from \eqref{c8} it can be seen that $J^1f\in L^2(\R)$.
From Lemma \ref{lemma}, for $\theta\in (0,1)$, we have that
\begin{equation}\label{interpp}
    \|J^\theta(\langle x\rangle^{(1-\theta)/2}f)\|_{2}\le c\|J^1f\|_2^\theta \|\langle x\rangle^{1/2}f\|_2^{1-\theta}<\infty.
\end{equation}
Taking $\theta=2\alpha$ we have that $\alpha-1/2+(1-\theta)/2=0$ and therefore we conclude $J^{2\alpha}v(t_*)\in L^2(\R)$, that is, \begin{equation}
    \label{ind} v(t_*)\in H^{2\alpha}(\R).
\end{equation} 

An iterative argument involving the proof of Theorem \ref{localtheorymkdv} with initial data $\tilde{v}_0=v(t_*)$ shows that $v\in C([0,T];H^{2\alpha}(\R))$. Moreover, using Theorem \ref{FLP2015} it can be seen that the fact $\langle x \rangle^{\alpha}v(t_*)\in L^2(\R)$ imply $v$ is in the class defined by \eqref{class1uno}-\eqref{class3tres} and \eqref{class4} with $Z_{2\alpha, \alpha}$ instead of $H^{2\alpha}(\R)$.

\subsection{Case $\alpha\in(1/2,1]$.}\hfill

Since $\alpha>1/2$, it can be seen that $|x|^{1/2}v(t_i)$ is in $L^2(\R)$ for $i=0,1$. Therefore, the conclusion of the previous case holds in $H^1(\R)$. In particular we know $v\in C\left([0,T];H^1(\R)\cap L^2(|x|dx) \right)$ with $\langle x\rangle^{\Tilde{\alpha}-1/2}\partial_x v$ in $L^2(\R)$ for almost every $t\in[0,t_1]$ and $\Tilde{\alpha} \in (0,1/2]$. We claim the latter now also holds for $\alpha\in(1/2,1]$.

To see this, let $\{v_{0m}\}_m\subset C_0^\infty(\R)$ be a sequence converging to $v_0$ in $Z_{1,1/2}$. Denote with $v_m(\cdot)\in H^\infty(\R)$ the corresponding solution provided by Theorem \ref{FLP2015} with initial data $v_{0m}$. By regularity of the data-solution map we can assume all the $v_m$'s are defined in $[0,T]$ with $v_m$ converging to $v$ in $C\left([0,T];Z_{1,1/2}\right)$. By following  the  same steps done to obtain \eqref{c8} we note that all remains equal except for the fact $|\phi_{0N}^\prime|$ is no longer bounded above independently of $N$. So, what is left is to estimate the terms involving $\phi_{N}^\prime$. In fact, instead of estimate \eqref{c3} we argue as follows:
\begin{equation*}
\begin{split}
    \int_0^{t_1}\int v_m^4\phi_N^\prime dxdt & \le c \int_0^{t_1} \|v_m\|^2_{L^\infty_x}\int v_m^2 \langle x\rangle dxdt\\ &\le c\|v_m\|^2_{L^\infty_{t_1}H^1}\|v_m\langle x\rangle^{1/2}\|^2_{L^2_{x{t_1}}}\le ct_1\|v_m\|^3_{L^\infty_{t_1}Z_{1,1/2}}\le 2c \|v\|^3_{L^\infty_{t_1}Z_{1,1/2}}.
\end{split}
\end{equation*}
Also, for estimate \eqref{c5} we note that $|\phi_N^\prime|\le c_N$ with $c_N$ depending on $N$. Since $N$ is fixed, the same computations remain valid.

Under these considerations, arguing as in \eqref{c0s}-\eqref{c7} it can be seen that for almost every $t\in [0,t_1]$ we have 
\begin{equation}
    \label{8s}\langle x\rangle^{\alpha-1/2}\partial_x v\in L^2(\R),
\end{equation}
which is our claim.
Moreover, a similar analysis shows that (without taking the limit as $m\to\infty$) for $m$ large enough
\begin{equation}\label{c9s}
    \|\langle x\rangle^{\alpha-1/2}\partial_x v_m\|_{L^2_{xt_1}}\le M,
\end{equation}
with $M$ depending on $\|v_0\|_{Z_{1,1/2}}$, $ \|\langle x\rangle^{\alpha-1/2}v(t_1)\|_{L^2_x}$ and $\|v\|_{L^\infty_{t_1}Z_{1,1/2}}$. Assume without loss of generality that \begin{equation}
    \label{abuse} \langle x \rangle^{\alpha-1/2} \partial_xv(t_i)\mbox{ are in } L^2(\R)\mbox{ for }i=0,1.
\end{equation}
Note in case \eqref{abuse} is not true, we can take a smaller subinterval $[t_0^*,t_1^*]\subset[0,t_1]$ in which the end points satisfy \eqref{abuse}.

Consider now $\phi_{N}\equiv \phi_{1,N}^\alpha$ build as done in \eqref{c0} but based on the function  \begin{equation*}
    \Tilde{\Tilde{\phi}}_{1,N}^\alpha (x):=
    \begin{cases}
        \langle x \rangle^{2\alpha-1}-1 \hspace{5mm}  x\in [0,N],\\
        (2N)^{2\alpha-1} \hspace{9mm}  x\in [3N,\infty).
    \end{cases}
\end{equation*}
Take $\partial_x$ to the mKdV equation and multiply it by $\partial_xv_m\phi_N$ to get
\begin{equation}
    \label{9b}
    \partial_t w_m w_m\phi_N+3\partial_x^3 w_m w_m\phi_N + 2v_mw_m^3\phi_N + v^2_mw_m\partial_x w_m \phi_N =0,
\end{equation}
where we denote $w_m:=\partial_x v_m$. Integrating by parts yields
\begin{equation*}
    \frac{1}{2}\frac{d}{dt}\int w_m^2 \phi_N dx +\frac{3}{2}\int (\partial_x w_m)^2 \phi_N^\prime dx -\frac{1}{2}\int w_m^2 \phi_N^{(3)}dx + \int 2v_mw_m^3\phi_N + \int v^2_mw_m\partial_x w_m \phi_N=0.
\end{equation*}
Note that, in terms of $w_m$, the first three terms above remain the same as in \eqref{c1} for the mKdV equation and the respective estimates are analogous. The only difference relies on the terms coming from the derivative of the nonlinearity $v^2_m\partial_x v_m$. Integrating by parts in the space variable we have that for $t\in[0,t_1]$:
\begin{equation*}
    2\int v_m w_m^3 \phi_Ndx=-2\int v_m^2w_m\partial_x w_m \phi_Ndx -\int v_m^2w_m^2\phi_N^\prime dx.
\end{equation*}
Since, for $m$ large enough,
\begin{equation}\label{c9bs}
    \left|\int v^2_m w_m^2 \phi_N^\prime dx \right|\le c\|v_m\|^2_{L^\infty_x}\|w_m\|^2_{L^2_x}\le c \|v_m\|^4_{L^\infty_{t_1} H^1}\le 2c\|v\|^4_{L^\infty_{t_1} H^1},    
\end{equation}
we only need to focus on the term $$\left|-\int_0^{t_1}\int v_m^2 w_m \partial_x w_m \phi_N dx dt \right|.$$
We proceed as follows: first we note that
\begin{equation}\label{c9}
    \begin{split}
        \left|-\int_0^{t_1}\int v_m^2 w_m \partial_x w \phi_N dx dt \right|&\le \|\partial_x w_m\|_{L^\infty_xL^2_{t_1}} \|w_m\phi_N^{1/2}\|_{L^2_{xt_1}} \|v_m^2 \phi_N^{1/2}\|_{L^2_xL^\infty_{t_1}}.
    \end{split}
\end{equation}
Let us estimate each term on right-hand side of \eqref{c9}. From the continuous dependence (see \eqref{class2dos} with $s=1$), for $m$ large enough, $$\|\partial_x^2v_m\|_{L^\infty_xL^2_{t_1}}\le 2c\|\partial_x^2v\|_{L^\infty_{x}L^2_{t_1}}. $$
Also, from \eqref{c9s} we know that,  for $m\gg1$, 
\begin{equation*}
   \|w_m\phi_N^{1/2}\|_{L^2_{xt_1}} \le c \|\langle x\rangle^{\alpha-1/2}\partial_x v_m\|_{L^2_{xt_1}}\le M.
\end{equation*}
Finally, for the term $\|v_m^2\phi_N^{1/2}\|_{L^2_xL^\infty_{t_1}}$ we note that $\|v_m^2\phi_N^{1/2}\|_{L^2_xL^\infty_{t_1}}\le \|v_m\phi_N^{1/4}\|^2_{L^4_xL^\infty_{t_1}}$.
Using the integral equation and the fact $|\phi_N|^{1/4}\le c\langle x\rangle^{3/8}$ we have that 
\begin{equation}\label{c10b}
    \|v_m\phi_N^{1/4}\|_{L^4_xL^\infty_{t_1}}\le \|\langle x\rangle^{3/8}U(t)v_{0m}\|_{L^4_xL^\infty_{t_1}}+\int_0^{t_1}\|\langle x\rangle^{3/8}U(t-t^\prime)v^2_m\partial_xv_m\|_{L^4_xL^\infty_{t_1}}dt^\prime .
\end{equation}
In view of \eqref{ww1} and \eqref{troca2} we obtain 
\begin{equation}\label{c10}
    \begin{split}
        \|\langle x\rangle^{3/8}U(t)v_{0m}\|_{L^4_xL^\infty_{t_1}}&\le \|U(t)\langle x\rangle^{3/8}v_{0m}\|_{L^4_xL^\infty_{t_1}} +\|U(t)\{\Phi_{t,3/8}(\widehat{u}_0)(\xi)\}^\vee\|_{L^4_xL^\infty_{t_1}}\\ &\le c\|D^{1/4}\langle x\rangle^{3/8}v_{0m}\|_{L^2_x} + c\|D^{1/4}\{\Phi_{t,3/8}(\widehat{u}_0)(\xi)\}^\vee\|_{L^2_x} \\ &\le c\|J^{1/4}\langle x\rangle^{3/8}v_{0m}\|_{L^2_x} + c\|D^{1/4}\{\Phi_{t,3/8}(\widehat{u}_0)(\xi)\}^\vee\|_{L^2_x}.
    \end{split}
\end{equation}
An interpolation using Lemma \ref{lemma} gives
\begin{equation}
    \label{c11} \|J^{1/4}_x \langle x\rangle^{3/8}f\|_{L^2_x}\le c\|J^1_xf\|^{1/4}_{L^2}\|\langle x\rangle^{1/2}f\|_{L^2}^{3/4}.
\end{equation}
Also, using \eqref{residualderiv} we get
\begin{equation}
    \label{c12} \|D^{1/4}\{\Phi_{t,3/8}(\widehat{f})(\xi)\}^\vee\|_{L^2_x}\le c(1+t_1)\|f\|_{H^1}.
\end{equation}
We combine \eqref{c11} and \eqref{c12} into \eqref{c10} to obtain, for $m$ large,
\begin{equation}\label{c13}
\begin{split}
    \|\langle x\rangle^{3/8}U(t)v_{0m}\|_{L^4_xL^\infty_{t_1}}& \le c \|J^1_xv_{0m}\|_{L^2_x}^{1/4}\|\langle x\rangle^{1/2}v_{0m}\|_{L^2_x}^{3/4}+c \|v_{0m}\|_{H^1}\\
    &\le 2c \|v_0\|_{Z_{1,1/2}}.
\end{split}
\end{equation}

For the other term in \eqref{c10b}, we argue in a similar fashion using \eqref{c11} and \eqref{c12} to get
\begin{equation}\label{c14}
    \begin{split}
        \int_0^{t_1}\|\langle x\rangle^{3/8}U(t-t^\prime)v_m^2\partial_x v_m\|_{L^4_xL^\infty_{t_1}}dt^\prime&\hfill \\&\hspace{-45mm}\le c\int_0^{t_1} \|J^1_x(v^2_m\partial_xv_m)\|^{1/4}_{L^2_x}\|\langle x\rangle^{1/2}v^2_m\partial_xv_m\|_{L^2_x}^{3/4}dt^\prime+c(1+{t_1})\int_0^{t_1}\|v^2_m\partial_xv_m\|_{H^1}dt^\prime \\
        &\hspace{-45mm} \le c\int_0^{t_1}\left( \|v^2_m\partial_xv_m\|_{H^1} + c \|\langle x\rangle^{1/2}v^2_m\partial_xv_m\|_{L^2_x}\right)dt^\prime.
    \end{split}
\end{equation}
According to \eqref{class2dos}, \eqref{class4} and \eqref{class4p}, for $m\gg1$,
\begin{equation}\label{c15b}
    \begin{split}
        \int_0^{t_1} \|\langle x\rangle^{1/2} v^2_m \partial_x v_m\|_{L^2_x}dt'&\le c\|\langle x\rangle^{1/2}\partial_x v_m v^2_m \|_{L^2_{xt_1}} \le c\|\langle x \rangle^{1/2}v_m\|_{L^5_{x}L^{10}_{t_1}} \|v_m\|_{L^4_xL^\infty_{t_1}} \|\partial_x v_m\|_{L^{20}_xL^{5/2}_{t_1}} \\& \le c \mu(v_m)^3 \le 2c\mu(v)^3,
    \end{split}
\end{equation}
where we used the continuous dependence on the initial data.  

Combining \eqref{localcomp1} and \eqref{c15b} it follows from \eqref{c14} that 
\begin{equation}\label{c15bb}
    \int_0^{t_1}\|\langle x\rangle^{3/8}U(t-t^\prime)v_m^2\partial_x v_m\|_{L^4_xL^\infty_{t_1}}dt^\prime \le M,
\end{equation}
where $M$ depends on several norms in which $v$ is known to be finite.
From \eqref{c10b}-\eqref{c15bb} we conclude 
\begin{equation}\label{c14s}
    \|v_m^2\phi_N^{1/2}\|_{L^2_xL^\infty_{t_1}}\le M.
\end{equation}

Gathering together \eqref{9b}-\eqref{c14s} we can emulate the argument in \eqref{c0s}-\eqref{c4} to conclude
\begin{equation*}
    \limsup\limits_{m\to\infty} \int_0^{t_1}\int(\partial_xw_m)^2\phi_N^\prime dxdt \le M,
\end{equation*}
with $M$ depending on $\|v_0\|_{Z_{1,1/2}}$ and  $\|\langle x\rangle^{\alpha}v(t_1)\|_{L^2_x}$ among other norms in which $v$ is known to be finite.

Note that the convergence argument done in \eqref{c5} can be emulated with $w_m$ instead of $v_m$ and using the continuous dependence in the norm $\|\partial^2_x v_m\|_{L^\infty_xL^2_{t_1}}$ provided by the local theory. 

Continuing as in \eqref{c6} and \eqref{c7} it can be seen that \begin{equation}
    \label{c15}\langle x\rangle^{\alpha-1}\partial_x w \in L^2(\R)\ \mbox{for almost every}\ t\in [0,t_1],
\end{equation}
where $w=\partial_{x}v$.
From \eqref{8s} and \eqref{c15} there exists $t_*\in[0,t_1]$ such that $\langle x\rangle^{1/2}f$ and $J^1f$ are in $L^2(\R)$, where $f$ denotes the function $\langle x\rangle^{\alpha-1}\partial_x v(t_*)$. Interpolating as in \eqref{interpp} with $\theta=2\alpha-1$ we conclude $v(t_*)\in H^{2\alpha}(\R)$. In addition, an iterative argument shows that $v\in C([0,T];H^{2\alpha}(\R))$. Moreover, it can be seen that $v$ is in the class defined by \eqref{class1uno}-\eqref{class3tres} and \eqref{class4} with $Z_{2\alpha, \alpha}$ instead of $H^{2\alpha}(\R)$.

\subsection{Case $\alpha\in(1,3/2]$.}\hfill

By taking  $\alpha=1$, applying the result of the previous case and the local theory in weighted spaces, we have that $v\in C\left([0,T];H^2(\R)\cap L^2(|x|^2dx) \right)$  with $\langle x\rangle^{1/2}\partial_x^k v(t)\in L^2(\R)$ for almost every $t\in [0,t_1]$ and $k=1,2$. 

We claim that $\langle x\rangle^{\alpha-1} \partial_x^2 v(t)$ is in $L^2(\R)$ for almost every $t\in[0,t_1]$. Indeed, let $\{v_{0m}\}_m\subset C_0^\infty(\R)$ be a sequence converging to $v_0$ in $Z_{2,1}$. Denote with $v_m(\cdot)\in H^\infty(\R)$ the corresponding solution provided by Theorem \ref{FLP2015} with initial data $v_{0m}$. By the regularity of the data-solution map again we may assume all the $v_m$'s are defined in $[0,T]$ with $v_m$ converging to $v$ in $C\left([0,T];Z_{2,1}\right)$. 

The idea is to emulate what was done from \eqref{9b} to \eqref{c15} noticing that $|\phi_{1,N}^\prime|$ is not longer uniformly bounded above independently of $N$. So, we re-estimate the related terms as follows:
instead of inequality \eqref{c9bs} we get
\begin{equation}\label{c16}
    \begin{split}
        \left|-\int_0^{t_1} \int v_m^2 (\partial_xv_m)^2\phi_{1N}^\prime dxdt  \right|& \le c\int_0^{t_1} \int v_m^2 (\partial_xv_m)^2 \langle x\rangle dxdt \le c\|\langle x\rangle^{1/2} v_m\|_{L^2_{x{t_1}}} \|\partial_x v_m\|_{L^{\infty}_{x{t_1}}}^2 \\
        &\le c \|\langle x \rangle^{1/2}v_m \|_{L^\infty_{t_1}L^2_x}\|v_m\|_{L^\infty_{t_1} H^2} \le c \|v_{m}\|^2_{Z_{2,1}}\le 2c \|v\|^2_{Z_{2,1}}.
    \end{split}
\end{equation}
Instead of \eqref{c9} we obtain
\begin{equation}
    \left|-\int_0^{t_1} \int v_m^2\partial_xv_m \partial_x^2v_m\phi_{1N} dxdt \right|\le \|\langle x\rangle^{3/2}v_m^2\|_{L^2_xL^\infty_{t_1}} \|\langle x\rangle^{1/2}\partial_x v_m\|_{L^2_{x{t_1}}} \|\partial_x^2 v_m\|_{L^\infty_xL^2_{t_1}}.
\end{equation}
Note that for $m$ large $\|\partial_x^2 v_m\|_{L^\infty_xL^2_{t_1}}\le c\|\partial_x^2 v\|_{L^\infty_xL^2_{t_1}}$, which is finite because of the local theory. Using \eqref{c9s} with $\alpha=1$, for $m\gg1$, it follows that $\|\langle x\rangle^{1/2}\partial_x v_m\|_{L^2_{xt_1}} \le M$, where $M$ depends on $\|v_{0}\|_{Z_{1,1/2}}$ and $\|\langle x\rangle^{1/2} v(t_1)\|_{L^2_x}$ among other norms for $v$. For the term $\|\langle x\rangle^{3/2}v_m^2\|_{L^2_xL^\infty_{t_1}}$ we have that 
\begin{equation}\label{c17}
    \begin{split}
        \|\langle x\rangle^{3/2}v_m^2\|_{L^2_xL^\infty_{t_1}}^{1/2}&\le\|\langle x\rangle^{3/4}v_m\|_{L^4_xL^\infty_{t_1}}\\&\le\|\langle x\rangle^{3/4}U(t)v_{0m}\|_{L^4_xL^\infty_{t_1}}+\int_0^{t_1} \|\langle x\rangle^{3/4}U(t-t^\prime)v^2_m \partial_xv_m\|_{L^4_xL^\infty_{t_1}} dt^\prime.
    \end{split}
\end{equation}
According to \eqref{ww1} and \eqref{troca2},
\begin{equation}
\begin{split}
    \|\langle x\rangle^{3/4}U(t)v_{0m}\|_{L^4_xL^\infty_{t_1}}&\le\|U(t)\langle x\rangle^{3/4}v_{0m}\|_{L^4_xL^\infty_{t_1}}+\|U(t)\{\Phi_{t,3/4}(\widehat{v}_{0m}(\xi)) \}^\vee\|_{L^4_xL^\infty_{t_1}}\\&\le c \|D^{1/4}_x \langle x\rangle^{3/4}v_{0m}\|_{L^2_x} + \|D^{1/4}_x \{\Phi_{t,3/4}(\widehat{v}_{0m}(\xi)) \}^\vee\|_{L^2_x}\\&\le c \|J^{1/4}_x \langle x\rangle^{3/4}v_{0m}\|_{L^2_x} + \|D^{1/4}_x \{\Phi_{t,3/4}(\widehat{v}_{0m}(\xi)) \}^\vee\|_{L^2_x}.
\end{split}
\end{equation}
Interpolating with $\theta=1/8$, it follows 
\begin{equation}
    \label{c18}
    \begin{split}
        \|J^{1/4}_x (\langle x\rangle^{3/4}v_{0m})\|_{L^2_x}&\le c\|J^2_{x}v_{0m}\|_{L^2_x}^{1/8}\|\langle x\rangle^{6/7}v_{0m}\|_{L^2_x}^{7/8}\\ &\le c \|v_{0m}\|_{H^2}+c\|\langle x\rangle v_{0m}\|_{L^2_x}\\ &\le  c\|v_{0m}\|_{Z_{2,1}}.
    \end{split}
\end{equation}
Also, by \eqref{residualderiv},
\begin{equation}\label{c19}
    \begin{split}
        \|D^{1/4}_x \{\Phi_{t,3/4}(\widehat{v}_{0m}(\xi)) \}^\vee\| \le c(1+t_1)\|v_{0m}\|_{H^{7/4}} \le c\|v_{0m}\|_{Z_{2,1}}.
    \end{split}
\end{equation}
Thus, for $m$ large enough $\|\langle x\rangle^{3/4}U(t)v_{0m}\|_{L^4_xL^\infty_{t_1}}\le 2c\|v_{0}\|_{Z_{2,1}}$.

For the remaining term in \eqref{c17}, we use \eqref{c11} and \eqref{c12} to get
\begin{equation}\label{c20b}
    \begin{split}
        \int_0^{t_1}\|\langle x\rangle^{3/4}U(t-t^\prime)v_m^2\partial_x v_m\|_{L^4_xL^\infty_{t_1}}dt^\prime&\hfill \\&\hspace{-45mm}\le c\int_0^{t_1} \|J^2_x(v^2_m\partial_xv_m)\|^{1/8}_{L^2_x}\|\langle x\rangle v^2_m\partial_xv_m\|_{L^2_x}^{7/8}dt^\prime+c(1+{t_1})\int_0^{t_1}\|v^2_m\partial_xv_m\|_{H^2}dt^\prime \\
        &\hspace{-45mm} \le c\int_0^{t_1}\left( \|v^2_m\partial_xv_m\|_{H^2} + c \|\langle x\rangle v^2_m\partial_xv_m\|_{L^2_x}\right)dt^\prime.
    \end{split}
\end{equation}
Using the continuous dependence on the initial data, for $m$ large enough, we have
\begin{equation}\label{c20bb}
    \begin{split}
        \int_0^{t_1} \|\langle x\rangle v^2_m \partial_x v_m\|_{L^2_x}dt'&\le c\|\langle x\rangle\partial_x v_m v^2_m \|_{L^1_{t_1}L^2_x} \le c\|\langle x \rangle^{1/2} \partial_x v_m\|_{L^2_{x{t_1}}} \|\langle x \rangle^{1/2}v_m\|_{L^\infty_{xt_1}} \|v_m\|_{L^2_{t_1}L^\infty_x} \\& \le c\|\langle x \rangle^{1/2} \partial_x v_m\|_{L^2_{x{t_1}}} \|v_m\|_{L^\infty_{t_1}H^1} \|J^1\langle x\rangle^{1/2} v_m\|_{L^\infty_{t_1}L^2_x}\\ &\le c \|\langle x \rangle^{1/2} \partial_x v_m\|_{L^2_{x{t_1}}} \|v_m\|_{L^\infty_{t_1}Z_{1,1/2}} (\|v_m\|_{L^\infty_{t_1}H^2}+\|\langle x\rangle v_m\|_{L^\infty_{t_1}L^2_x}) \le  M,
    \end{split}
\end{equation}
where we interpolated with $\theta=1/2$ and used \eqref{c9s} with $\alpha=1$. Note $M$ depends on $\|v\|_{L^\infty_{t_1}Z_{2,1}}$ among other norms in which $v$ is known to be finite. 
Combining \eqref{localcomp1} and \eqref{c20bb} it follows that, for $m\gg1$,
\begin{equation}\label{c20}
    \int_0^{t_1}\|\langle x\rangle^{3/4}U(t-t^\prime)v^2_m\partial_xv_m\|_{L^4_xL^\infty_{t_1}}dt^\prime \le M,
\end{equation}
where $M$ depends on $\|v_0\|_{Z_{2,1}}$ and  $\|\langle x\rangle^{\alpha}v(t_1)\|_{L^2_x}$ among other norms in which $v$ is known to be finite.

From \eqref{c17}-\eqref{c20} we conclude
\begin{equation}\label{c20ss}
    \|\langle x\rangle^{3/2}v_m^2\|_{L^2_xL^\infty_{t_1}}^{1/2}\le\|\langle x\rangle^{3/4}v_m\|_{L^4_xL^\infty_{t_1}}\le M.
\end{equation}
Under these two modifications, an analogous argument as the one developed in \eqref{9b}-\eqref{c15} supports that for almost every $t\in[0,t_1]$ we have 
\begin{equation}\label{c21}
    \langle x\rangle^{\alpha-1} \partial_x^2 v(t)\in L^2(\R),
\end{equation}
which proves our claim.

Note also that, if in the proof of \eqref{c21} the convergence in $m$ is skipped, we may see that for $m\gg1$
\begin{equation}
    \label{c21ss}
    \|\langle x\rangle^{\alpha-1} \partial_x^2 v_m\|_{L^2_{xt_1}}\le M,
\end{equation} with $M$ depending on norms in which $v$ is finite.

In what follows, without loss of generality, we assume  that $\langle x \rangle^{\alpha-1} \partial_x^2v(t_i)$ are in $L^2(\R)$ for $i=0,1$. Next we claim that $\langle x\rangle^{\alpha-3/2}\partial_x^3v\in L^2(\R)$ for almost every $t\in[0,t_1]$.

Indeed, denote $w_m:=\partial_x^2v_m$ and define $\phi_N\equiv \phi_{2,N}^\alpha$ as done in \eqref{c0} based on 
\begin{equation*}
    \Tilde{\Tilde{\phi}}_{2,N}^\alpha (x):=
    \begin{cases}
        \langle x \rangle^{2\alpha-2}-1 \hspace{5mm}  x\in [0,N],\\
        (2N)^{2\alpha-2} \hspace{9mm}  x\in [3N,\infty).
    \end{cases}
\end{equation*}
Taking $\partial_x^2$ to the mKdV equation and multiplying it by $w_m\phi_N$ we have
\begin{equation}
    \label{c22} \partial_tw_mw_m\phi_N +\partial_x^3w_m w_m\phi_N + \left(2(\partial_xv_m)^3+6v_m\partial_xv_mw_m+v_m^2\partial_xw_m \right)w_m\phi_N=0.
\end{equation}
Integrate \eqref{c22} by parts in space, after an extra integration over $[0,t_1]$, we arrive to
\begin{equation}\label{c23}
\begin{split}
    \int& w_{m}^2(0)\phi_N dx -\int w_m^2(t_1)\phi_N dx + 3\int_0^{t_1}\int (\partial_xw_m)^2\phi_N^\prime dxdt -\int_0^{t_1}\int w_m^2\phi_N^{(3)}dxdt+\\&4\int_0^{t_1}\int (\partial_x v_m)^3w_m\phi_Ndxdt+12\int_0^{t_1}\int v_m\partial_xv_m w_m^2\phi_Ndxdt +2\int_0^{t_1}\int v_m^2\partial_x w_mw_m\phi_Ndxdt=0
\end{split}
\end{equation}

Note \begin{equation}
    \left|4\int_0^{t_1}\int (\partial_xv_m)^3 w_m \phi_N dxdt \right| \le c\|\langle x\rangle^{1/2}\partial_x v_m\|_{L^2_{xt_1}}\|\langle x\rangle^{\alpha-1}w_m\|_{L^2_{xt_1}}\|\partial_xv_m\|_{L^{\infty}_{x{t_1}}}^2,
\end{equation}
where the terms $\|\langle x\rangle^{1/2}\partial_x v_m\|_{L^2_{xt_1}}$ and $\|\langle x\rangle^{\alpha-1} \partial_x^2v_m\|_{L^2_{xt_1}}$ were estimated in \eqref{c9s} and \eqref{c21ss}. The term $\|\partial_xv_m\|_{L^\infty_{x{t_1}}}$ may be estimated via Sobolev embedding.

We conclude for $m$ large enough that \begin{equation}
    \label{c24} \left|\int_0^{t_1}\int (\partial_xv_m)^3w_m\phi_Ndxdt \right|\le M,
\end{equation}
where $M$ depends on several norms related to $v$ such as $\|v_{0}\|_{Z_{2,1}}$ and $\|\langle x\rangle^{1/2}\partial_xv(t_1)\|_{L^2_x}$.

For the next term, using \eqref{c21ss}, we have that for $m$ large
\begin{equation}\label{c25}
    \begin{split}
        \left|12\int_0^{t_1}\int v_m \partial_xv_m w_m^2\phi_N dxdt \right|&\le c\|v_m\|_{L^\infty_{x{t_1}}}\|\partial_xv_m\|_{L^\infty_{x{t_1}}} \|w_m^2\phi_N\|_{L^1_{xt_1}}\\&\le c\|v_m\|_{L^\infty_{t_1}H^2}^2 \|\langle x\rangle^{\alpha-1}w_m\|^2_{L^2_{xt_1}} \le M.
    \end{split}
\end{equation}
On the other hand,
\begin{equation*}
    \left|2\int_0^{t_1} \int v_m^2 \partial_x w_m w_m\phi_Ndxdt \right|\le \|\partial_xw_m\|_{L^\infty_xL^2_{t_1}}\|\langle x\rangle v_m^2\|_{L^2_xL^\infty_{t_1}}\|w_m\|_{L^2_{x{t_1}}}.
\end{equation*}
By the local theory we have that for $m\gg1$ $$\|\partial_xw_m\|_{L^\infty_xL^2_{t_1}}\le c \|\partial^3_x v_m\|_{L^\infty_xL^2_{t_1}}\le 2c\|v\|_{L^\infty_{t_1}H^2}.$$
and $$\|w_m\|_{L^2_{x{t_1}}}\le c\|v\|_{L^\infty_{t_1}Z_{2,1}}.$$
Finally, for $\|\langle x\rangle v_m^2\|_{L^2_xL^\infty_{t_1}}\le\|\langle x\rangle^{1/2}v_m\|_{L^4_xL^\infty_{t_1}}^2$ we note
\begin{equation}
    \|\langle x\rangle^{1/2}v_m\|_{L^4_xL^\infty_{t_1}}\le \|\langle x\rangle^{3/4}v_m\|_{L^4_xL^\infty_{t_1}}\le M,
\end{equation}
where we used \eqref{c20ss}. We conclude that for $m$ large enough,
\begin{equation}
    \label{c26}
     \left|2\int_0^{t_1} \int v_m^2 \partial_x w_m w_m\phi_Ndxdt \right|\le M.
\end{equation}

Combining \eqref{c24}-\eqref{c26} into \eqref{c23} and arguing as in \eqref{c0s}-\eqref{c7} it can be seen that \begin{equation*}
    \int_0^{t_1}\int (\partial_x w)^2 \langle x\rangle^{2\alpha-2} dx <\infty,
\end{equation*}
that is, \begin{equation}\label{c27}
    \langle x\rangle^{\alpha-3/2}\partial_x^3v(t)\in L^2(\R) \mbox{ for almost every }t\in[0,t_1].
\end{equation}
Using \eqref{c21} and \eqref{c27} it follows via interpolation that for some $t_*\in[0,t_1]$ we have \begin{equation}
    \label{c27ss} v(t_*)\in H^{2\alpha}(\R).
\end{equation}
The latter implies $v\in C\left( [0,T];H^{2\alpha}(\R)\cap L^2(|x|^{2\alpha}dx)\right)$ as before.

\subsection{Case $\alpha\in(r/2,(r+1)/2]$, $r\ge3$.}\hfill

It is enough to prove by induction that for any $n\in\{2,3,\dots,r\}$ it follows that for all $\beta$ in $(\frac{n}{2},\frac{n+1}{2}]$ with $\beta\le\alpha$, we have 
\begin{equation}\label{induc}
    \begin{cases}
        \langle x\rangle^{\beta-n/2}\partial_x^n v(t) \in L^2(\R), \\ 
        \langle x\rangle^{\beta-n/2-1/2}\partial_x^{n+1} v(t) \in L^2(\R),
    \end{cases}
\end{equation}
for almost every $t$ in a closed subinterval $I$ of $[0,t_1]$.
In such case, taking $n=r$ and $\beta=\alpha$ and interpolating as in \eqref{interpp}, there would exists $t_{*}\in [0,t_1]$ such that $v(t_{*})\in H^{2\alpha}(\R)$, which leads to the desired conclusion.

Note the ``base'' case ($n=2$) follows from \eqref{c21} and \eqref{c27}. 
For the inductive step, assume \eqref{induc} is valid for $n-1$. We will prove it is also valid for $n$. 
Using the inductive hypothesis with $\beta=n/2$, it can be seen that, for $f:=\langle x\rangle^{\beta-n/2+1/2}\partial_x^n v$, we have $J^1f$ and $\langle x\rangle f$ are in $L^2(\R)$.
Interpolating as in \eqref{interpp}, there exists $t_{**}\in [0,t_1]$ such that $v(t_{**})\in H^{n}(\R)$, which  implies that $v\in C([0,T];H^n(\R))$. Moreover, since $\alpha>n/2$, using Theorem \ref{FLP2015} it can be seen that $v\in C([0,T];Z_{n,n/2})$. 

Recall, via Sobolev embedding, that for $k=1,2,\dots,n-1$: \begin{equation}
    \label{c29} \|\partial_x^k v\|_{L^\infty_x}\le c\|v\|_{H^n}.
\end{equation} 
In addition, using Lemma \ref{HA2} and interpolation, for any $t\in[0,T]$ we have that
\begin{equation}
    \label{c30}\|\langle x\rangle^j\partial_x^kv\|_{L^2_x}\le c\|v\|_{H^n}^\theta \|\langle x\rangle^{n/2}v\|^{1-\theta}_{L^2_x}\le c\|v\|_{Z_{n,n/2}},
\end{equation}
where $\theta=k/n$ and $j=(n-k)/2$.

Without loss of generality assume $I=[0,t_1]$. 
Let $\beta\in(\frac{n}{2},\frac{n+1}{2}]$ with $\beta \le \alpha$. We claim that $\langle x\rangle^{\beta-n/2}\partial_x^n v(t) \in L^2(\R)$ for almost every $t\in[0,t_1]$. In fact, let $\{v_{0m}\}_m\subset C_0^\infty(\R)$ be a sequence converging to $v_0$ in $Z_{n,n/2}$. Denote with $v_m$ the corresponding solution provided by Theorem \ref{FLP2015} with initial data $v_{0m}$, which  in view of the regularity of the data-solution map may be assumed to be defined in $[0,T]$ and $v_m$ converges to $v$ in $C\left([0,T];Z_{n,n/2}\right)$. 

Consider $\phi_N\equiv\phi_{(n-1),N}^\beta$ built as in \eqref{c0} from the function 
\begin{equation*}
    \Tilde{\Tilde{\phi}}_{(n-1),N}^\beta (x):=
    \begin{cases}
        \langle x \rangle^{2\beta-n+1}-1, \hspace{5mm}  x\in [0,N],\\
        (2N)^{2\beta-n+1}, \hspace{9mm}  x\in [3N,\infty).
    \end{cases}
\end{equation*}
Taking $\partial_x^{n-1}$ to the mKdV equation for $v_m$ we get $$\partial_t z_m + \partial_x^3 z_m + \partial_x^{n-1}(v^2_m\partial_x v_m)=0, \hspace{3mm} z_m:=\partial_x^{n-1}v_m.$$ We focus only on the nonlinear terms. Since $$\partial_x^{n-1}(v^2_m\partial_x v_m)=\sum\limits_{k\le n-1}\sum\limits_{j\le k}\binom{n-1}{k}\binom{k}{j}\partial_x^jv_m \partial_x^{k-j}v_m \partial_x^{n-k}v_m,$$
 after multiplication by $\partial_x^{n-1}v_m\phi_N$ and integration we get
 \begin{equation}\begin{split}
     \label{c31b} \int_0^{t_1}\int \partial_{x}^{n-1}v_m &\partial_x^{n-1}(v^2_m\partial_xv_m)\phi_Ndxdt\\&=\sum\limits_{k\le n-1}\sum\limits_{j\le k}\binom{n-1}{k}\binom{k}{j}\int_0^{t_1}\int \partial_x^{n-1}v_m \partial_x^{n-k}v_m \partial_x^j v_m \partial_x^{k-j}v_m \phi_N dxdt. \\ &=: \sum\limits_{k\le n-1}\sum\limits_{j\le k}\binom{n-1}{k}\binom{k}{j} A_{k,j}.
     \end{split}
 \end{equation}
We now estimate the terms $A_{k,j}$. It is necessary to split into some cases:

\noindent \textbf{Case A1, if $k=0$}. 
\begin{equation}\label{c31}
\begin{split}
    |A_{0,0}|&\le \|\partial_x^{n-1}v_m\|_{L^\infty_{x{t_1}}}\|v_m\phi_N^{1/4}\|_{L^\infty_{x{t_1}}}\|\partial_x^n v_m\|_{L^2_{x{t_1}}}\|v_m\phi_N^{3/4}\|_{L^2_{x{t_1}}}\\&\le c\|v_m\|_{L^\infty_{t_1}H^2}^2\|J^1\langle x\rangle^{1/2}v_m\|_{L^\infty_{t_1}L^2_x}\|\langle x\rangle^{3/2}v_m\|_{L^\infty_{t_1}L^2_x}.
    \end{split}
\end{equation}
Note that interpolating with $\theta=1/n$ we have
\begin{equation}\label{rhcp}\begin{split}
    \|J^1\langle x\rangle^{1/2}v_m\|_{L^\infty_{t_1}L^2_x}&\le 
    c\|v_m\|_{L^\infty_{t_1}H^n}^\theta \|\langle x \rangle^{n/(2(n-1))}v_m\|_{L^\infty_{t_1}L^2_x}^{1-\theta}\le c\|v_{m}\|_{L^\infty_{t_1}Z_{n,n/2}}. 
    \end{split}
\end{equation}
Using that $3/2\le n/2$, \eqref{c31} and \eqref{rhcp} we conclude that for $m$ large enough
\begin{equation}
    \label{c32} |A_{0,0}|\le c\|v\|^{4}_{L^\infty_{t_1}Z_{n,n/2}}.
\end{equation}

\noindent \textbf{Case A2, if $k\neq0$}.

Note in this case, since $n\ge3$ we have $|\phi_N|\le c\langle x\rangle^{2}\le c\langle x\rangle^{n-k/2}$. Using this and \eqref{c29}-\eqref{c30} we have, for any $j\in \{0,1,\dots,n-1\}$ and $m\gg1$, that
\begin{equation}\label{c33}
    \begin{split}
        |A_{k,j}|&\le c\|\partial_x^{n-1}v_m\|_{L^\infty_{x{t_1}}} \|\partial_x^{n-k}v_m\|_{L^\infty_{x{t_1}}} \|\langle x\rangle^{(n-j)/2}\partial_x^jv_m\|_{L^2_{x{t_1}}} \|\langle x\rangle^{(n-k+j)/2}\partial_x^{k-j}v_m\|_{L^2_{x{t_1}}}\\&\le c\|v_m\|_{L^\infty_{t_1}H^n}^2\|v_m\|^{2\theta}_{L^\infty_{t_1} H^n} \|\langle x\rangle^{n/2}v_m\|_{L^\infty_{t_1} L^2_x}^{2-2\theta} \le c\|v\|^4_{L^\infty_{t_1}Z_{n,n/2}}.
    \end{split}
\end{equation}
From \eqref{c32} and \eqref{c33} and arguing for $z_m$ analogous to what was done in \eqref{c0s}-\eqref{c8} for $v_m$ we see that
\begin{equation}\label{c35}
    \langle x\rangle^{\beta-n/2} \partial_x^n v(t) \in L^2(\R) \mbox{ for almost every }t\in[0,t_1],
\end{equation}
which proves our claim.

By arguing in a similar fashion, we also obtain that for $m$ large enough
\begin{equation}\label{c35ss}
    \|\langle x\rangle^{\beta-n/2} \partial_x^n v_m(t)\|_{L^2_{xt_1}} \le M,
\end{equation}
where $M$ depends on $\|v\|_{L^\infty_{t_1}Z_{n,n/2}}$, $\|v_0\|_{Z_{n,n/2}}$ and $\|\langle x\rangle^{\beta-n/2} z_m(t_1)\|_{L^2_x}$.

From \eqref{c35} there exist $t_{0}^*<t_{1}^*$ so that $\langle x\rangle^{\alpha-n/2}\partial_x^{n}v(t_{0}^*)$ and $\langle x\rangle^{\alpha-n/2}\partial_x^{n}v(t_{1}^*)$ are in $L^2(\R)$. For simplicity (and without loss of generality) we denote $[t_{0}^*,t_{1}^*]$ with $[0,t_1]$ again.

Next, we claim that $\langle x\rangle^{\beta-(n+1)/2}\partial_x^{n+1}v(t)$ is in $L^2(\R)$ for almost every $t\in[0,t_1]$. Indeed, construct $\phi_N\equiv\phi_{n,N}^\beta$ as in \eqref{c0} but based on the function 
\begin{equation*}
    \Tilde{\Tilde{\phi}}_{n,N}^\beta (x):=
    \begin{cases}
        \langle x \rangle^{2\beta-n}-1, \hspace{5mm}  x\in [0,N],\\
        (2N)^{2\beta-n}, \hspace{9mm}  x\in [3N,\infty).
    \end{cases}
\end{equation*}
Denote $w_m:=\partial_x^n v_m$. Take $\partial_x^n$ to the mKdV equation for $v_m$ and multiply it by $w_m\phi_N$. Again, we will take care only on the nonlinear term, which is this case is 
\begin{equation*}
    \begin{split}
        \int_0^{t_1}\int w_m\partial_x^n(v^2_m\partial_xv_m)\phi_Ndxdt&=\sum\limits_{k\le n}\sum\limits_{j\le k}\binom{n}{k}\binom{k}{j}\int_0^{t_1}\int w_m \partial_x^{n+1-k}v_m \partial_x^jv_m\partial_x^{k-j}v_m\phi_Ndxdt \\ &\equiv\sum\limits_{k\le n}\sum\limits_{j\le k}\binom{n}{k}\binom{k}{j}B_{k,j}.
    \end{split}
\end{equation*}

We now split into three cases:

\noindent \textbf{Case B1, if $k=0$}. 

Using \eqref{class2dos}, \eqref{c14s} and \eqref{c35ss} we have that for $m$ large enough
\begin{equation*}
    \begin{split}
        |B_{0,0}|&\le c\|\partial_x^{n+1}v_m\|_{L^\infty_xL^2_{t_1}}\|w_m\phi_N^{1/2}\|_{L^2_{xt_1}}\|v_m^2\phi_N^{1/2}\|_{L^2_xL^\infty_{t_1}}\\&\le c\|\partial_x^{n+1} v_m\|_{L^\infty_xL^2_{t_1}}\|\langle x\rangle^{\beta-n/2}w_m\|_{L^2_{xt_1}}\|v_m \langle x \rangle^{1/4}\|^2_{L^4_xL^\infty_{t_1}}\le M,
    \end{split}
\end{equation*}
where $M$ depends on $\|v\|_{L^\infty_{t_1}Z_{n,n/2}}$ among other norms in which $v$ is finite. 

\noindent \textbf{Case B2, if $k=n$}. 

Note that in case $j=0$ or $j=n$, using \eqref{c35ss} we have for $m\gg1$
\begin{equation*}\begin{split}
    |B_{n,n}|&=|B_{n,0}|\le c\|\partial_x v_m\|_{L^\infty_{x{t_1}}}\|v_m\|_{L^\infty_{x{t_1}}}\|\langle x\rangle^{\beta-n/2}w_m\|^2_{L^2_{xt_1}}\\&\le c \|v_m\|^2_{L^\infty_{t_1}H^n} \|\langle x\rangle^{\beta-n/2}w_m\|^2_{L^2_{xt_1}} \le M
    \end{split}
\end{equation*}
Also, if $j\in\{1,\dots,n-1\}$ we use \eqref{c29}, \eqref{c30} and \eqref{c35} to get for $m$ large
\begin{equation*}
    \begin{split}
        |B_{n,j}|&\le c \|\langle x\rangle^{\beta-n/2}w_m\|^2_{L^2_{xt_1}}\|\langle x\rangle^{1/2}\partial_xv_m\|_{L^2_{x{t_1}}}\|\partial_x^jv_m\|_{L^\infty_{x{t_1}}} \|\partial_x^{n-j}v_m\|_{L^\infty_{x{t_1}}}\\&\le c \|v_m\|^2_{L^\infty_{t_1}H^n} \|\langle x\rangle^{\beta-n/2}w_m\|^2_{L^2_{xt_1}} \|v_m\|_{L^\infty_{t_1}Z_{n,n/2}}\\&\le M.
    \end{split}
\end{equation*}

\noindent\textbf{Case B3, if $k\in\{1,\dots,n-1\}$}.

In this case we argue in a similar manner as done in Case B2. Namely, for $m\gg1$,
\begin{equation*}
    \begin{split}
        |B_{k,j}|&\le c \|\langle x\rangle^{\beta-n/2}w_m\|^2_{L^2_{xt_1}}\|\langle x\rangle^{\beta-n/2}\partial_x^{n+1-k}v_m\|_{L^2_{xt_1}}\|\partial_x^jv_m\|_{L^\infty_{x{t_1}}} \|\partial_x^{k-j}v_m\|_{L^\infty_{x{t_1}}}\\&\le c \|v_m\|^2_{L^\infty_{t_1}H^n} \|\langle x\rangle^{\beta-n/2}w_m\|^2_{L^2_{xt_1}} \|\langle x\rangle^{\beta-n/2}\partial_x^{n+1-k}v_m\|_{L^2_{xt_1}}\\&\le M.
    \end{split}
\end{equation*}

With these estimates of the nonlinear term in hand, it can be seen (as in \eqref{c0s}-\eqref{c7}) that
\begin{equation*}
    \int_0^{t_1} \int (\partial_x^{n+1}v)^2\langle x\rangle^{2\beta-(n+1)}dxdt<\infty,
\end{equation*}
which implies \begin{equation}\label{c36}
    \langle x\rangle^{\beta-(n+1)/2}\partial_x^{n+1}v(t)\in L^2(\R)\mbox{ for almost every }t\in[0,t_1],
\end{equation}
which establishes our claim. From \eqref{c35} and \eqref{c36} the inductive step follows, and the proof of the theorem is completed.
\end{proof}

Note that, under minor modifications, the proof of Theorem \ref{teo0} remains valid for $v_0$ in $H^s(\R)$ for any $s\ge 1/4$. In fact, the following corollary is a direct consequence of this fact.

\begin{corollary}\label{color}
    Let $v_0\in H^s(\R)$, $s\ge 1/4$. Let $v$ be the global in-time solution provided by Theorem \ref{global}. Assume there exist $t_0, t_1 \in \R$ and $\alpha>0$ such that 
     \begin{equation*}
     |x|^{\frac{s}{2}+\alpha} v(t_0)\in L^2(\R) \mbox{ and } |x|^{\frac{s}{2}+\alpha} v(t_1) \in L^2(\R). 
\end{equation*} 
Then $v\in C\left( [-T,T]; H^{s+2\alpha}(\R)\right)$, for any $T>0$.
\end{corollary}
\begin{proof}
  Without loss of generality assume $T>\max\{|t_0|,|t_1|\}$. Note that $u_0\in H^{{\frac{1}{4}}}(\R)$ with the respective solution $v\in C([-T,T]; H^{{\frac{1}{4}}}(\R))$ so that $|x|^{\tilde{\alpha}}v(t_i)$ are in $L^2(\R)$ for $i=0,1$ and $\tilde{\alpha}=\frac{s}{2}+\alpha>0$. By Theorem \ref{teo0} we obtain that $v\in C([-T,T]; H^{2\tilde{\alpha}}(\R))$. The result then follows since $2\tilde{\alpha}=s+2\alpha$.
\end{proof}

\section*{Acknowledgments}
This work was financed in part by the Coordenação de Aperfeiçoamento de Pessoal de N\'ivel Superior - Brasil (CAPES) - Finance Code 001. A.P. was partially supported by Conselho Nacional de Desenvolvimento Cient\'ifico e Tecnol\'ogico - Brasil (CNPq) grant 2019/02512-5 and Fundação de Amparo à Pesquisa do Estado de São Paulo - FAPESP grant 2019/02512-5.

\section*{data availability statement}

 All data generated or analysed during this study are  available within the article.


\end{document}